\documentclass[12pt]{article}

\usepackage{graphicx}
\usepackage{amsmath,amsthm,amssymb}
\usepackage{mathrsfs}

\usepackage{hyperref}

\usepackage[left=2cm,right=2cm,top=3.5cm,bottom=3.5cm]{geometry}

\usepackage{color}
\usepackage{array}

\theoremstyle{plain}
\newtheorem{thm}{Theorem}[section]
\newtheorem{prop}[thm]{Proposition}

\newtheorem{cor}[thm]{Corollary}
\theoremstyle{definition}
\newtheorem{defn}[thm]{Definition}

\newtheorem*{ex*}{Example}
\theoremstyle{remark}
\newtheorem*{rem}{Remark}

\newcommand{\vr}{\varrho}
\newcommand{\vc}[1]{{\bf #1}}
\newcommand{\vu}{\vc{u}}
\newcommand{\vm}{\vc{m}}
\newcommand{\vB}{\vc{B}}
\newcommand{\vx}{\vc{x}}
\newcommand{\Div}{{\rm div}\,}
\newcommand{\Curl}{{\rm curl}\,}
\newcommand{\Grad}{\nabla}
\newcommand{\dx}{\,{\rm d}\vc{x}}
\newcommand{\dt}{\,{\rm d}t}

\newcommand{\dr}{\,{\rm d}r}
\newcommand{\intO}[1]{\int_{\Omega}#1\dx\,}
\newcommand{\intQ}[1]{\int_{Q}#1\dx\,}
\newcommand{\intQi}[1]{\int_{Q_i}#1\dx\,}
\newcommand{\DC}{C^\infty_{\rm c}}
\newcommand{\R}{\mathbb{R}}
\newcommand{\bfphi}{\boldsymbol{\varphi}}

\newcommand{\half}{\frac{1}{2}}
\newcommand{\trans}{^{\mathsf{T}}}

\begin{document}

\title{Non--uniqueness of entropy--conserving solutions to the ideal compressible MHD equations} 

\author{Christian Klingenberg \and Simon Markfelder}

\date{\today}

\maketitle


\centerline{Department of Mathematics, W\"urzburg University}

\centerline{Emil-Fischer-Str. 40, 97074 W\"urzburg, Germany}

\bigskip

\begin{abstract}
In this note we consider the ideal compressible magneto--hydrodynamics (MHD) equations in a special two dimensional setting. We show that there exist particular initial data for which one obtains infinitely many entropy--conserving weak solutions by using the convex integration technique. Our result is also true for the isentropic case.
\end{abstract}


\section{Introduction} 

We consider the ideal compressible magneto--hydrodynamics (MHD) equations 
\begin{equation} \label{MHD}
	\begin{split}
		\partial_t \vr + \Div (\vr\vu) &=0,\\
		\partial_t (\vr\vu) + \Div (\vr\vu \otimes \vu) + \Grad p - (\Curl \vB )\times \vB &= 0,\\
		\partial_t \left(\half \vr |\vu|^2 + \vr e(\vr,p) + \half |\vB|^2 \right) + \Div \left[\left(\half \vr |\vu|^2 + \vr e(\vr,p) + p + |\vB|^2 \right) \vu\right] - \Div \big((\vB\cdot \vu) \vB\big) &= 0, \\
		\partial_t \vB + \Curl (\vB\times \vu) &= 0,\\
		\Div \vB &=0.
	\end{split}
\end{equation}
The unknown functions in \eqref{MHD} are the density $\vr > 0$, the pressure $p>0$, the velocity $\vu\in\R^3$ and the magnetic field $\vB\in\R^3$, which are all functions of the time $t\in[0,T)$ and the spatial variable $\vx=(x,y,z)\trans\in\R^3$. The internal energy $e$ is a given function of the density $\vr$ and the pressure $p$. 

In this note we consider a special two dimensional setting. Let $\Omega\subset\R^2$ a bounded two dimensional spacial domain. We consider $\vu=(u,v,0)\trans$ and $\vB=(0,0,b) \trans$ and furthermore we let all the unknowns only depend on $(x,y)\in\Omega$. From now on we write $\vu=(u,v)\trans\in\R^2$ and $\vx=(x,y)\trans\in\Omega\subset\R^2$ for the corresponding two dimensional vectors. Then the MHD system \eqref{MHD} turns into 
\begin{equation} \label{MHD.setting}
\begin{split}
	\partial_t \vr + \Div (\vr\vu) &=0,\\
	\partial_t (\vr\vu) + \Div \big(\vr\vu \otimes \vu\big) + \Grad \Big( p + \half b^2\Big) &= 0,\\
	\partial_t \left(\half \vr |\vu|^2 + \vr e(\vr,p) + \half  b^2 \right) + \Div \left[\left(\half \vr |\vu|^2 + \vr e(\vr,p) + p + b^2 \right) \vu \right] &= 0, \\
	\partial_t b + \Div (b\vu) &= 0.
\end{split}
\end{equation} 
Note that in \eqref{MHD.setting} $\Div,\Grad$ are two dimensional differential operators in contrast to \eqref{MHD}, where they are three dimensional differential operators.

We endow system \eqref{MHD.setting} with initial conditions  
\begin{equation} \label{initial}
	\big(\vr,p,\vu,b\big)(0,\cdot) = \big(\vr_0,p_0,\vu_0,b_0\big)
\end{equation}
and impermeability boundary conditions
\begin{equation} \label{boundary}
	\vu\cdot \vc{n}\big|_{\partial \Omega} = 0.
\end{equation}

\begin{defn} \label{defn:weak}
	A 4-tuple $(\vr,p,\vu,b)\in L^\infty \big([0,T)\times \Omega; (0,\infty)\times(0,\infty)\times \R^2 \times \R\big)$ is a weak solution to \eqref{MHD.setting}, \eqref{initial}, \eqref{boundary} if the following equations hold for all test functions $\varphi,\phi,\psi\in \DC\big([0,T)\times\R^2\big)$ and $\bfphi\in\DC\big([0,T)\times\R^2;\R^2\big)$ with $\bfphi\cdot \vc{n}\big|_{\partial \Omega}$:
	\begin{equation} \label{weak1}
		\int_0^T\intO{\big[\vr \partial_t\varphi + \vr\vu\cdot\Grad\varphi\big]}\dt + \intO{\vr_0 \varphi(0,\cdot)} = 0 
	\end{equation}
	\\
	\begin{equation} \label{weak2}
		\int_0^T\intO{\left[\vr\vu\cdot \partial_t\bfphi + \big(\vr\vu\otimes\vu \big) : \Grad\bfphi + \Big(p + \half b^2\Big)\Div\bfphi\right]}\dt + \intO{\vr_0 \vu_0\cdot\bfphi(0,\cdot)} = 0
	\end{equation}
	\\
	\begin{equation} \label{weak3}
		\begin{split}
		\int_0^T\intO{\bigg[\Big(\half \vr |\vu|^2 + \vr e(\vr,p) + \half b^2 \Big)\partial_t\phi 
		+ \Big(\half \vr |\vu|^2 + \vr e(\vr,p) + p + b^2 \Big) \vu \cdot\Grad\phi\bigg]}\dt \quad& \\
		+ \intO{\Big(\half \vr_0 |\vu_0|^2 + \vr_0 e(\vr_0,p_0) + \half b_0^2 \Big) \phi(0,\cdot)} &= 0 
		\end{split}
	\end{equation}
	\\
	\begin{equation} \label{weak4}
		\int_0^T\intO{\big[b\partial_t\psi + b\vu\cdot\Grad\psi\big]}\dt + \intO{b_0 \psi(0,\cdot)} = 0
	\end{equation} 
\end{defn} 

\begin{rem}
	The impermeability boundary condition is represented by the choice of the test functions.
\end{rem}

\begin{rem}
	Note that we exclude vacuum for our consideration, i.e. in this note $\vr>0$, $p>0$.
\end{rem} 

It is a well--known fact that there may exist physically non--relevant weak solutions to conservation laws. Hence one has to introduce additional selection criteria in order to single out the physically relevant weak solutions. A common approach is to impose an entropy inequality. However for the MHD system \eqref{MHD} there is no known entropy. 

Note that for the Euler system the functions $$\eta = -\vr s(\vr,p) \quad\text{ and }\quad q= -\vr s(\vr,p) \vu$$     
form an entropy pair. Here the specific entropy $s=s(\vr,p)$ is a given function as well as the internal energy $e$ and note that these functions are interrelated by Gibbs' relation.

It is a straightforward computation to show that a strong solution to the MHD system \eqref{MHD} fulfills
\begin{equation} \label{eneq}
\partial_t \big(\vr s(\vr,p)\big) + \Div \big(\vr s(\vr,p)  \vu\big) = 0.
\end{equation}
Although this suggests that $(\eta,q)$ is an entropy pair for the MHD system, too, $(\eta,q)$ is \emph{not} an entropy pair for MHD, cf. \cite{ChaKli16}. However $(\eta,q)$ is still used as a selection criterion in the literature for example if Riemann problems are considered and one wants to find out whether or not a shock is physical, see e. g. \cite{Torrilhon03}. We misuse terminology and call $\eta$ and $q$ still entropy, entropy flux respectively. 

The weak solutions, whose existence we will prove in this note, fulfill the entropy equation \eqref{eneq} in the weak sense. We call such solutions entropy--conserving.

\begin{defn} \label{defn:conserving} 
	A weak solution $(\vr,p,\vu,b)$ to \eqref{MHD.setting}, \eqref{initial}, \eqref{boundary} is called \emph{entropy--conserving}, if for all test functions $\varphi\in\DC\big([0,T)\times\R^2\big)$ the entropy equation 
	\begin{equation} \label{conserving}
		\int_0^T\intO{\big[\vr s(\vr,p) \partial_t \varphi + \vr s(\vr,p)\vu\cdot \Grad\varphi\big]}\dt + \intO{\vr_0 s(\vr_0,p_0) \varphi(0,\cdot)} = 0
	\end{equation}
	holds.
\end{defn}

The following theorem is our main result: 
\begin{thm} \label{thm:main}
	Let $\vr_0,p_0\in L^\infty(\Omega;(0,\infty))$ and $b_0\in L^\infty(\Omega)$ be arbitrary piecewise constant functions. Then there exists $\vu_0\in L^\infty (\Omega;\R^2)$ such that there are infinitely many entropy--conserving weak solutions to \eqref{MHD.setting} with initial data $\vr_0,p_0,\vu_0,b_0$ and impermeability boundary condition. These solutions have the property that $\vr$, $p$ and $b$ do not depend on time; in other words $\vr\equiv \vr_0$, $p\equiv p_0$ and $b\equiv b_0$.
\end{thm}

The proof of Theorem \ref{thm:main} relies on the non--uniqueness proof for the full Euler system provided in \cite{FKKM17} and consists of two main ideas. The first one is to make use of a result (see Proposition \ref{prop:convint} below) which was proved by Feireisl \cite{Feireisl16} and also by Chiodaroli \cite{Chiodaroli14}. This result is based on the convex integration method, that was developed by De Lellis and Sz\'ekelyhidi \cite{DelSze09,DelSze10} in the context of the pressureless incompressible Euler equations. The second idea is the fact that $\vr$, $p$ and $b$ can be chosen \emph{piecewise} constant, what was observed originally by Luo, Xie and Xin \cite{LuoXieXin16}.

Note that non--uniqueness of weak solutions fulfilling an entropy inequality (even in one space dimension) is well--known: There exist Riemann initial data for which one has more than one solutions, see e. g. Torrilhon \cite{Torrilhon03} and references therein. 

Note furthermore that there is also a convex integration result to incompressible ideal MHD by Bronzi et al. \cite{BroLopNus15}. There the same two dimensional setting is considered as in the present note. In contrast to this note, where a convex integration result for Euler is used, Bronzi et al. apply the convex integration techique directly to an incompressible version of \eqref{MHD.setting}.

\section{Proof of the main result}

In order to prove Theorem \ref{thm:main} we will make use of the following proposition whose proof is based on convex integration. 

\begin{prop} \label{prop:convint}
	Let $Q\subset \R^2$ a bounded domain, $\vr>0$ and $C>0$ positive constants. Then there exists $\vm_0 \in L^\infty(Q; \R^2)$ such that there are infinitely many functions
	\begin{equation*} 
		\vm \in L^\infty\big((0,T) \times Q; \R^2\big) \cap C_{\rm weak}\big([0,T]; L^2(Q; \R^2)\big)
	\end{equation*}
	satisfying
	\begin{align}
		\int_0^T \intQ{ \vm \cdot \Grad \varphi }\dt &= 0 \label{ci1} \\ 
		\int_0^T \intQ{ \left[ \vm \cdot \partial_t \bfphi + \left( \frac{ \vm \otimes \vm}{\vr} - \frac{1}{2} \frac{|\vm|^2}{\vr} \mathbb{I} \right)
		: \Grad \bfphi \right] }\dt + \intQ{ \vm_0 \cdot \bfphi(0,\cdot) } &= 0 \label{ci2}
	\end{align}
	for all test functions $\varphi \in \DC ([0,T] \times \R^2)$ and $\bfphi \in \DC([0,T) \times \R^2; \R^2)$,
	and additionally
	\begin{align*}
		E_{\rm kin} = \frac{1}{2} \frac{|\vm|^2}{\vr} &= C \quad \text{ a.e. in }(0,T) \times Q,&\
		E_{0,\rm kin} = \frac{1}{2} \frac{|\vm_0|^2}{\vr} &= C \quad \text{ a.e. in }Q.
	\end{align*}
\end{prop}
For the proof of Proposition \ref{prop:convint} we refer to \cite[Theorem 13.6.1]{Feireisl16}.

Now we are able to prove Theorem \ref{thm:main}. 
\begin{proof}
	Let $\vr_0,p_0\in L^\infty(\Omega;(0,\infty))$ and $b_0\in L^\infty(\Omega)$ given piecewise constant functions. Then there exist finitely many $Q_i\subset \Omega$ open and pairwise disjoint, such that $\Omega=\bigcup\limits_{i} \overline{Q_i}$ and $\vr_0\big|_{Q_i}=\vr_i$, $p_0\big|_{Q_i}=p_i$ and $b_0\big|_{Q_i}=b_i$ with constants $\vr_i,p_i>0$ and $b_i\in\R$.	We apply Proposition \ref{prop:convint} on each $Q_i$ to $\vr=\vr_i$ and $C=\Lambda - p_i - \half b_i^2$, where $\Lambda$ is a constant with $\Lambda>\max\limits_{i}\left(p_i + \half b_i^2\right)$. This yields $\vm_{0,i}\in L^\infty(Q_i;\R^2)$ and infinitely many $\vm_i\in L^\infty((0,T) \times Q_i; \R^2)$ with the properties given in Proposition \ref{prop:convint}. We then piece together the $\vm_{0,i}\in L^\infty(Q_i;\R^2)$ to $\vm_{0}\in L^\infty(\Omega;\R^2)$ and the $\vm_{i}\in L^\infty((0,T) \times Q_i; \R^2)$ to $\vm\in L^\infty((0,T) \times \Omega; \R^2)$. 
	
	We define $\vu_0:=\frac{\vm_0}{\vr_0} \in L^\infty(\Omega;\R^2)$ and for each $\vm$ we define a function $\vu:=\frac{\vm}{\vr_0} \in L^\infty\big((0,T) \times \Omega; \R^2\big)$. Furthermore we define $(\vr,p,b)\in L^\infty\big([0,T)\times \Omega; (0,\infty)\times(0,\infty) \times \R\big)$ by $\vr\equiv\vr_0$, $p\equiv p_0$ and $b\equiv b_0$. We claim that $(\vr,p,\vu,b)$ is an entropy--conserving weak solution to \eqref{MHD.setting} with initial data $\vr_0,p_0,\vu_0,b_0$.
	
	Let $\varphi,\phi,\psi\in \DC\big([0,T)\times\R^2\big)$ and $\bfphi\in\DC\big([0,T)\times\R^2;\R^2\big)$ with $\bfphi\cdot\vc{n}\big|_{\partial\Omega}$ arbitrary test functions. 
	Using \eqref{ci1} and \eqref{ci2}, we obtain the following.
	\begin{align*} 
		&\int_0^T\intO{\big[\vr \partial_t\varphi + \vr\vu\cdot\Grad\varphi\big]}\dt + \intO{\vr_0 \varphi(0,\cdot)} \\
		&= \sum\limits_i \vr_i \intQi{\left(\int_0^T \partial_t\varphi \dt + \varphi(0,\cdot)\right)} + \sum\limits_{i}\int_0^T\intQi{\vm_i\cdot\Grad\varphi }\dt \ \ =\ \ 0
	\end{align*}
	\\
	\begin{align*} 
		&\int_0^T\intO{\left[\vr\vu\cdot \partial_t\bfphi + \big(\vr\vu\otimes\vu \big) : \Grad\bfphi + \Big(p + \half b^2\Big)\Div\bfphi\right]}\dt + \intO{\vr_0 \vu_0\cdot\bfphi(0,\cdot)} \\
		&= \sum\limits_{i} \Bigg(\int_0^T \intQi{ \left[ \vm_i \cdot \partial_t \bfphi + \left( \frac{ \vm_i \otimes \vm_i}{\vr_i} - \frac{1}{2} \frac{|\vm_i|^2}{\vr_i} \mathbb{I} \right)
			: \Grad \bfphi \right] }\dt + \intQi{ \vm_{0,i} \cdot \bfphi(0,\cdot) } \Bigg) \\
		&\qquad + \sum\limits_{i} \int_0^T \intQi{  \left[\half \frac{|\vm_i|^2}{\vr_i} + \Big(p_i + \half b_i^2\Big)\right]\Div\bfphi }\dt \\
		&= \Lambda \int_0^T \intO{\Div\bfphi }\dt \ \ =\ \  0
	\end{align*}
	\\
	\begin{align*} 
		&\int_0^T\intO{\bigg[\Big(\half \vr |\vu|^2 + \vr e(\vr,p) + \half b^2 \Big)\partial_t\phi 
		+ \Big(\half \vr |\vu|^2 + \vr e(\vr,p) + p + b^2 \Big) \vu \cdot\Grad\phi\bigg]}\dt \\
		&\qquad+ \intO{\Big(\half \vr_0 |\vu_0|^2 + \vr_0 e(\vr_0,p_0) + \half b_0^2 \Big) \phi(0,\cdot)} \\
		&= \sum\limits_i \Big(\Lambda + \vr_i e(\vr_i,p_i) - p_i\Big) \intQi{\left(\int_0^T \partial_t\phi \dt + \phi(0,\cdot)\right)} \\
		&\qquad + \sum\limits_{i}\frac{\Lambda + \vr_i e(\vr_i,p_i) + \half b_i}{\vr_i}\int_0^T\intQi{\vm_i\cdot\Grad\phi }\dt \ \ =\ \ 0
	\end{align*}
	\\
	\begin{align*} 
		&\int_0^T\intO{\big[b\partial_t\psi + b\vu\cdot\Grad\psi\big]}\dt + \intO{b_0 \psi(0,\cdot)}  \\
		&= \sum\limits_i b_i \intQi{\left(\int_0^T \partial_t\psi \dt + \psi(0,\cdot)\right)} + \sum\limits_{i}\frac{b_i}{\vr_i}\int_0^T\intQi{\vm_i\cdot\Grad\psi }\dt \ \ =\ \ 0
	\end{align*} 

	We have shown that the equations \eqref{weak1} - \eqref{weak4} hold. Hence $(\vr,p,\vu,b)$ is indeed a weak solution. It remains to show that this solution is entropy--conserving. In other words we have to show that \eqref{conserving} holds. Let $\varphi\in\DC\big([0,T)\times\R^2\big)$ be an arbitrary test function. We obtain
	\begin{align*} 
		&\int_0^T\intO{\big[\vr s(\vr,p) \partial_t \varphi + \vr s(\vr,p)\vu\cdot \Grad\varphi\big]}\dt + \intO{\vr_0 s(\vr_0,p_0) \varphi(0,\cdot)} \\
		&= \sum\limits_i \vr_i s(\vr_i,p_i) \intQi{\left(\int_0^T \partial_t\varphi \dt + \varphi(0,\cdot)\right)} + \sum\limits_{i} s(\vr_i,p_i) \int_0^T\intQi{\vm_i\cdot\Grad\psi }\dt \ \ =\ \ 0.
	\end{align*}
	Thus $(\vr,p,\vu,b)$ is an entropy--conserving weak solution. Since there are infinitely many $\vm$ from Prop. \ref{prop:convint}, there are infinitely many entropy--conserving solutions $(\vr,p,\vu,b)$.
\end{proof}

\section{Isentropic MHD} 

In this section we extend our result to isentropic MHD equations. The isentropic MHD system reads 
\begin{equation} \label{isen.MHD}
\begin{split}
	\partial_t \vr + \Div (\vr\vu) &=0,\\
	\partial_t (\vr\vu) + \Div (\vr\vu \otimes \vu) + \Grad p(\vr) - (\Curl \vB )\times \vB &= 0,\\
	\partial_t \vB + \Curl (\vB\times \vu) &= 0,\\
	\Div \vB &=0.
\end{split}
\end{equation}
The unknown functions are the density $\vr>0$, the velocity $\vu\in\R^3$ and the magnetic field $\vB\in\R^3$. In contrast to the MHD system \eqref{MHD} the pressure $p$ in \eqref{isen.MHD} is not an unknown but a given function of the density, where $p(\vr)>0$ for all $\vr>0$. 

Again we consider a two dimensional setting. Let $\Omega\subset\R^2$ a bounded two dimensional spacial domain. We consider $\vu=(u,v,0)\trans$ and $\vB=(0,0,b) \trans$ and furthermore we let all the unknowns only depend on $(x,y)\in\Omega$. Then the isentropic MHD system \eqref{isen.MHD} turns into 
\begin{equation} \label{isen.MHD.setting}
\begin{split} 
	\partial_t \vr + \Div (\vr\vu) &=0,\\
	\partial_t (\vr\vu) + \Div \big(\vr\vu \otimes \vu\big) + \Grad \Big( p(\vr) + \half b^2\Big) &= 0,\\
	\partial_t b + \Div (b\vu) &= 0.
\end{split}
\end{equation} 

For the isentropic \emph{Euler} system, the energy $$\eta = \half \vr |\vu|^2 + P(\vr) + \half |\vB|^2$$ is an entropy. Here $P(\vr)$ is called pressure potential and is given by 
$$
P(\vr)=\vr\int_1^\vr\frac{p(r)}{r}\dr.
$$
Similar to the full MHD system considered above, one can show that the energy is \emph{not} an entropy for \eqref{isen.MHD} but strong solutions fulfill the corresponding energy equation
\begin{equation} \label{isen.enineq}
	\partial_t \left(\half \vr |\vu|^2 + P(\vr) + \half |\vB|^2 \right) + \Div \left[\left(\half \vr |\vu|^2 + P(\vr) + p(\vr) + |\vB|^2 \right) \vu\right] - \Div \big((\vB\cdot \vu) \vB\big) = 0. 
\end{equation}

Hence we will look for energy--conserving weak solutions. In the considered setting the energy equation \eqref{isen.enineq} turns into
\begin{equation*} 
	\partial_t \left(\half \vr |\vu|^2 + P(\vr) + \half b^2 \right) + \Div \left[\left(\half \vr |\vu|^2 + P(\vr) + p(\vr) + b^2 \right) \vu\right]  = 0. 
\end{equation*}

\begin{defn} \label{defn:isen.weak}
	A triple $(\vr,\vu,b)\in L^\infty \big([0,T)\times \Omega; (0,\infty)\times \R^2 \times \R\big)$ is a weak solution to \eqref{isen.MHD.setting} with initial data $\vr_0,\vu_0,b_0$ and impermeability boundary condition if the following equations hold for all test functions $\varphi,\psi\in \DC\big([0,T)\times\R^2\big)$ and $\bfphi\in\DC\big([0,T)\times\R^2;\R^2\big)$ with $\bfphi\cdot \vc{n}\big|_{\partial \Omega}$:
	\begin{equation} \label{isen.weak1}
		\int_0^T\intO{\big[\vr \partial_t\varphi + \vr\vu\cdot\Grad\varphi\big]}\dt + \intO{\vr_0 \varphi(0,\cdot)} = 0 
	\end{equation}
	\\
	\begin{equation} \label{isen.weak2}
		\int_0^T\intO{\left[\vr\vu\cdot \partial_t\bfphi + \big(\vr\vu\otimes\vu \big) : \Grad\bfphi + \Big(p(\vr) + \half b^2\Big)\Div\bfphi\right]}\dt 		+ \intO{\vr_0 \vu_0\cdot\bfphi(0,\cdot)} = 0
	\end{equation}
	\\
	\begin{equation} \label{isen.weak3}
	\int_0^T\intO{\big[b\partial_t\psi + b\vu\cdot\Grad\psi\big]}\dt + \intO{b_0 \psi(0,\cdot)} = 0
	\end{equation} 
	
	A weak solution is called energy--conserving if in addtion for all test functions $\phi\in\DC\big([0,T)\times\R^2\big)$ the energy equation 
	\begin{equation} \label{isen.conserving} 
	\begin{split}
		\int_0^T\intO{\bigg[\Big(\half \vr |\vu|^2 + P(\vr) + \half b^2 \Big)\partial_t\phi 
			+ \Big(\half \vr |\vu|^2 + P(\vr) + p(\vr) + b^2 \Big) \vu \cdot\Grad\phi\bigg]}\dt \quad& \\
		+ \intO{\Big(\half \vr_0 |\vu_0|^2 + P(\vr_0) + \half b_0^2 \Big) \phi(0,\cdot)} &= 0 
	\end{split}
	\end{equation}
	holds.
\end{defn}

The Cauchy problem for the isentropic MHD equations is ill--posed, too:

\begin{cor} \label{isen.MainResult}
	Let $\vr_0\in L^\infty(\Omega;(0,\infty))$ and $b_0\in L^\infty(\Omega)$ be arbitrary piecewise constant functions. Then there exists $\vu_0\in L^\infty (\Omega;\R^2)$ such that there are infinitely many energy--conserving weak solutions to \eqref{isen.MHD.setting} with initial data $\vr_0,\vu_0,b_0$ and impermeability boundary condition. These solutions have the property that $\vr$ and $b$ do not depend on time; in other words $\vr\equiv \vr_0$ and $b\equiv b_0$.
\end{cor}

\begin{proof} 
	Let $\vr_0\in L^\infty(\Omega;(0,\infty))$ and $b_0\in L^\infty(\Omega)$ given piecewise constant functions. Set furthermore $p_0:=p(\vr_0)$. Then $p_0\in L^\infty(\Omega;(0,\infty))$ is a piecewise constant function. Additionally we can choose the function $e(\vr,p)$ in such a way that $\vr_0 e(\vr_0,p_0)=P(\vr_0)$. We know from Theorem \ref{thm:main} that there exists an initial velocity $\vu_0\in L^\infty(\Omega;\R^2)$ such that there are infinitely many entropy--conserving weak solutions $(\vr\equiv\vr_0,p\equiv p_0,\vu,b\equiv b_0)$ to \eqref{MHD.setting} with initial data $\vr_0,p_0,\vu_0,b_0$. It is easy to check that for each of these solutions, the triple $(\vr\equiv\vr_0,\vu,b \equiv b_0)$ is an energy--conserving weak solution to the isentropic MHD equations \eqref{isen.MHD.setting} with initial data $\vr_0,\vu_0,b_0$ in the sense of Definition \ref{defn:isen.weak}.
\end{proof}

\section*{Acknowledgement}
The authors thank Bruno Despres for fruitful discussions.



\end{document}